\documentclass[12pt]{article}
\usepackage[utf8]{inputenc}
 
\usepackage{amssymb}
\usepackage{authblk}
\usepackage{amsmath}
\usepackage{amsthm}
\usepackage{url}
\usepackage{mathrsfs}
\usepackage{tikz-cd}

\usepackage{hyperref}

\theoremstyle{definition}
\newtheorem{mytheor}{Theorem}

\newtheorem{prop}{Proposition}[section]
\newtheorem{lem}{Lemma}[section]
\newtheorem{deff}{Definition}[section]
\newtheorem{theor}{Theorem}[section]
\newtheorem{cor}{Corolary}[section]
\newtheorem{quest}{Question}[section]
\newtheorem{exampl}{Example}[section]
\newtheorem{rem}{Remark}[section]

\newcommand{\N}{\mathbb{N}}
 
\newcommand{\Z}{\mathbb{Z}}

\newcommand{\meas}{\mathcal{M}}
\newcommand{\invmeas}{\mathcal{IM}}
\newcommand{\join}{\mathcal{J}}

\newcommand{\expect}{\mathbb{E}}

\DeclareMathOperator{\proj}{proj}

\DeclareMathOperator{\pord}{pOrd}
\DeclareMathOperator{\ord}{tOrd}
\DeclareMathOperator{\ext}{Ext}

\DeclareMathOperator{\gpSL}{SL}
\DeclareMathOperator{\ordext}{OrdExt}
\DeclareMathOperator{\orb}{Orb}

\DeclareMathOperator{\sml}{sml}
\newcommand{\smlp}{\sml_\sqsubset^+}
\newcommand{\smlm}{\sml_\sqsubset^-}
\newcommand{\smlbp}{\sml_{\overline\sqsubset}^+}
\newcommand{\smlbm}{\sml_{\overline\sqsubset}^-}
\newcommand{\bsmlp}{\overline{\sml}_\sqsubset^+}
\newcommand{\bsmlm}{\overline{\sml}_\sqsubset^-}
\newcommand{\bsqsub}{\overline{\sqsubset}}

\newcommand{\acts}{\curvearrowright}

\newcommand{\wand}{\mathcal{W}}

\newcommand{\repa}{\alpha}

\DeclareMathOperator{\lift}{Lift}

\begin{document}

\author{Andrei Alpeev\footnote{Euler Mathematical Institute at St. Petersburg State University\\alpeevandrey@gmail.com}}
\title{The invariant random order extension property is equivalent to amenability}

\maketitle
\begin{abstract}
Recently, Glasner, Lin and Meyerovitch gave a first example of a partial invariant order on a certain group that cannot be invariantly extended to an invariant random total order. Using their result as a starting point we prove that any invariant random partial order on a countable group could be invariantly extended to an invariant random total order iff the group is amenable.
\end{abstract}

\section{Introduction}

A {\em partial order} on a set $X$ is a binary relation $\prec$ that is:
\begin{enumerate}
\item {\em transitive}: $x \prec y$ and $y \prec z$ imply $x \prec z$,  for every $x,y,z \in X$;
\item {\em antisymmetric}: $x \prec y$ implies that not $y \prec x$,  for every $x, y \in X$;
\item {\em antireflexive}: not $x \prec x$, for every $x \in X$. 
\end{enumerate}
A {\em total(or linear) order} is an order that satisfies an additional requirement that 
\begin{enumerate}
\setcounter{enumi}{3}
\item {\em total}: $x \neq y$ implies either $x \prec y$ or $y \prec x$, for all $x,y \in X$.
\end{enumerate}
For any set $X$ we denote $\ord(X)$ the set of all total orders on $X$ and $\pord(X)$ the set of all partial orders on $X$. Note that if $X$ is a countable set, then both $\ord(X)$ and $\pord$ are metrizable compact sets (both are closed subsets of $2^{X \times X}$). For a partial order $\sqsubset \in \pord(X)$ we denote $\ext(\sqsubset)$ the set of all total orders on $X$ that extend $\sqsubset$, i.e. $x \sqsubset y$ implies $x \prec y$ for all $\prec \in \ext(\sqsubset)$ and $x,y \in X$. We denote $\ordext(X) \subset \pord(X) \times \ord(X)$ the set of all pairs $(\omega, \omega')$ such that $\omega \in \pord(X)$ and $\omega' \in \ext(\omega)$. Note that $\ordext(X)$ is a closed subset of $\pord(X) \times \ord(X)$.

Let $G$ be a countable group. There are two commuting left actions of $G$ on $\pord(G)$.  One defined by 
$x (g \prec) y$ iff $g^{-1}x \prec g^{-1} y$, we will call it the L-action. Another, we will cal it the R-action (but a left action nonetheless) is defined by $x (g\prec) y$ iff $xg \prec yg$. For the most part we will work with the R-action, unless otherwise is stated. A {\em right-invariant order (or a right-order)} on $G$ is an order which is fixed under the R-action of $G$. A group that admits a linear right-invariant order is called right-orderable. Orderings of groups form an old and fruitful area of research, we refer to \cite{KoM96}, \cite{Gla99} for classical results and to \cite{DNR14} for newer.

A {\em (right) invariant random partial order}(IRO) is a Borel probability measure $\nu$ on $\pord(G)$ that is invariant under the R-action of $G$. 

Invariant random orders were fruitfully used in entropy theory of measure preserving actions as a substitute for deterministic random orders (in particular, to produce a ``random past''), see \cite{Ki75}, \cite{StTZ80}, \cite{AMR21}, \cite{DOWZ21}. On the other hand, the use of dynamical methods in the order theory is presented in e.g. \cite{LM14}, \cite{DNR14}.

On every countable group there is a non-trivial invariant random linear order:
\begin{exampl}
Let $G$ be a countable group and let $G \acts [0,1]^G$ be the Bernoulli action with the unit interval endowed with the standard Lebesgue measure as a base space. For a random realization $(w_g)_{g \in G}$ of the corresponding process we assign an order on $G$ given by $g \prec h$ iff $w_g < w_h$, for $g,h \in G$. It is easy to check that the distribution of the order-valued random variable we described is indeed an invariant random liner order. 
\end{exampl}

More generally, the action of a countable group on its space of  linear orders is universal in the category of essentially free  actions on the standard probability space as was shown in \cite{GLM22}[Corollary 5.4]. 

Note that $\ord(G)$ is a closed invariant subset of $\pord(G)$ and $\ordext(G)$ is a closed invariant subset of $\pord(G) \times \ord(G)$ thus $\ordext(G)$ is endowed with a $G$-action.
\begin{deff}
An {\em extension} of the invariant random order $\nu$ is an invariant Borel probability measure $\lambda$ on $\ordext(G)$ such that $\proj_{\pord(G)}(\lambda) = \nu$. We say that a group has the {\em IRO-extension property} if for every invariant random order on $G$ there is an extension.
\end{deff}

Classical Rhemtulla-Formanek theorem \cite{R72}, \cite{F73} states that any partial right order on a torsion-free nilpotent group could be extended to a total right-order. This result fails already for torsion-free metabelian groups, see \cite{R72}. For invariant random orders extension is possible in a much larger class of groups.

\begin{theor}[\cite{St78}, \cite{AMR21}, Lemma 2.1]\label{thm: amenable gp has iro}
A countable amenable group has the IRO-extension property. 
\end{theor}

The natural question is whether all groups have this property. A counterexample was obtained recently:

\begin{theor}[Glasner-Lin-Meyerovitch \cite{GLM22}]
$\gpSL_3(\Z)$ has no IRO-extension property.
\end{theor}
Our main result is the following:
\begin{mytheor}\label{thm: main group}
A countable non-amenable group has no IRO-extension property.
\end{mytheor} 
Thus we have a characterization:
\begin{cor}
A countable group has the IRO-extension property iff it is amenable.
\end{cor}
Our first step in that direction is to consider the case of a free group $F_2$ in Section \ref{sec: free}:
\begin{mytheor}\label{thm: free no iro}
The free group $F_2$ has no IRO-extension property. In fact, there is a (deterministic) invariant partial order on $F_2$ that cannot be extended to an invariant random total order. 
\end{mytheor}

Let $S$ be the subsemigroup in $\gpSL_3(\Z)$ consisting of matrices in the form $I + A$, where $I$ is the identity operator matrix and $A$ is a matrix with all non-negative entries and at least one strictly positive entry.
In their example \cite{GLM22}, Glasner, Lin and Meyerovitch showed that for the left-invariant order $\sqsubset$ on $\gpSL_3(\Z)$ generated by subsemigroup\footnote{any proper subsemigroup $S$ that does not contain the group identity, generates a partial left-order by $a \prec b$ iff $e \prec a^{-1}b$ iff $a^{-1}b \in S$} S, the set of extending total orders $\ext(\sqsubset)$ has no $G$-invariant probability measure. The technique used is a deep extention upon the argument of Witte-Morris  from \cite{W94} regarding the non-orderability of finite-index subgroups of $\gpSL_n(\Z)$.

The idea of our proof is to consider a factorization $\pi: F_2 \to \gpSL_3(\Z)$. Note that for any map $\varphi: X \to Y$ there is a corresponding map $\varphi^{Ord} : \pord(Y) \to \pord(X)$ defined by $x_1 (\varphi^{Ord} \prec) x_2$ iff $\varphi(x_1) \prec \varphi(x_2)$. It seems natural to assume that any counterexample to the IRO-extension property could be lifted over a factorization, but I don't know whether this is the case.
\begin{quest}
Let $\pi: G \to H$ be an epimorphism between groups. Is it true that for any IRO $\nu$ on $H$ that can't be extended, the IRO $\pi^{Ord}(\nu)$ on $G$ couldn't be extended also? What if $\nu$ is a $\delta$-measures, that is a deterministic invariant partial order?
\end{quest}
To tackle this problem we carefully tweak the construction from \cite{GLM22} to show that $\ext(\pi^{Ord}(\sqsubset))$ does not carry an invariant probability measure.

\begin{cor}
Any group that contains a nonabelian free subgroup, has a left-invariant partial order $\prec$ such that $\ext(\prec)$ does not carry an invariant probability measure.
\end{cor}

\begin{rem}
The proof of Theorem \ref{thm: free no iro} actually gives that for any (not necessarily free) group $F$ that projects onto $\gpSL_3(\Z)$, the lift of the order $\sqsubset$, defined above, cannot be extended to an invariant random linear order on $F$.
\end{rem}

\begin{quest}
Which groups admit a left-invariant order $\prec$ such that $\ext(\prec)$ does not have a left-invariant probability measure?
\end{quest}

By the famous Olshanski\u{\i}  counterexample \cite{Ol80}, \cite{Ol91} to the von Neumann-Day conjecture, there are non-amenable groups without nonabelian free subgroups, so the corolary above does not cover Theorem \ref{thm: main group}. To circumvent this issue we introduce, in Section \ref{sec: orders on equiv}, the notion of an IRO on an equivalence relation and utilize the Gaboriau-Lyons theorem to deduce Theorem \ref{thm: main group} from Theorem \ref{thm: free no iro}. We remind that the Gaboriau-Lyons theorem \cite{GL09} states that the orbit equivalence relation of the Bernoulli action $G \acts [0,1]^G$ for non-amenable group $G$ contains the orbit equivalence relation of an essentially free action of the free group $F_2$. We show that the IRO-extension property for a group and for an orbit equivalence relation of its free actions are equivalent (Lemma \ref{lem: iro group and equiv}), that the IRO-extension property passes to subequivalence relations (Lemma \ref{lem: iro equiv hereditary}). This suffices to prove Theorem \ref{thm: main group}. This type of technique was used by Monod and Ozawa \cite{MO09} to prove that lamplighters over non-amenable groups are not unitarizable, and by myself to prove that lamplighter over non-amenable groups are not strongly Ulam stable \cite{A22}.

Finally, we use a generalization of the Gaboriau-Lyons theorem to equivalence relations by Bowen, Hoff and Ioana \cite{BHI18} to prove a characterization of the IRO-extension property for Borel equivalence relations similar to Theorem \ref{thm: main group}. We refer the reader to section \ref{sec: orders on equiv} for the definition of an invariant random order on a pmp countable Borel equivalence relation and corresponding invariant random order extension property.

\begin{mytheor}\label{thm: equiv IRO}
A pmp countable Borel ergodic equivalence relation on a standard probability space has the IRO-extension property iff it is amenable.
\end{mytheor}

Trivially, a non-ergodic equivalence relation has the IRO-extension property iff almost all its ergodic components have this property.

We have a large supply of groups and equivalence relations with invariant random orders that cannot be extended invariantly into total orders. We may define a maximal invariant random order, that is a partial invariant order that cannot be extended. Formally, let us define the extension order on the space of IRO's on the group $G$. We say that IRO $\nu_1$ is smaller than IRO $\nu_2$ in the extension order if there is a $G$-invariant Borel probability measure on $\pord(G) \times \pord(G)$ those marginals are $\nu_1$ and $\nu_2$ respectively and that is supported on the set of pairs of orders $(\prec_1, \prec_2)$ with $\prec_1$ being contained in $\prec_2$ (in other words, $x \prec_1 y$ implies $x \prec_2 y$ for all $x,y \in G$).
\begin{prop}
Every IRO on a countable group $G$ is smaller in the extension order than some extension-maximal IRO.
\end{prop}
\begin{proof}
We would like to apply Zorn's lemma, so we need to establish that every monotone chain has an upper bound. To this end we first show that every monotone chain has a finite subchain or a chain isomorphic to the first countable ordinal $\omega$. 
	
We define the following continuous function $f$ on $\pord(G)$:
	
\[
f(\prec) = \sum_{i \in \N} \frac{1}{2^i} \mathbb{I}_{x_i \prec y_i},
\]
for all $\prec \in \pord(G)$, where $(x_i,y_i)$ is some enumerations of all the pairs from $G \times G$, and $\mathbb{I}_{x_i \prec y_i} = 1$ iff $x_i \prec y_i$ and is $0$ otherwise. Observe that if $\nu_2$ is bigger than $\nu_1$ in the extension order then $\int f d\nu_1 \leq \int f d\nu_2$ and the equality is attained exactly when $\nu_1 = \nu_2$. 
Let $M$ be the supremum of $\int f d\nu$ where $\nu$ runs over all the elements of the chain under consideration. We simply take $\nu_1$ to be any element of the chain and recursively define $\nu_{i+1}$ to be an IRO that is not smaller in the extension order than $\nu_i$ and such that $\int f d\nu_{i+1} > M - 1/i$. It is easy to observe that $(\nu_i)$ either stabilizes or indeed forms a cofinal chain that is isomorphic to the $\omega$ ordinal. If it stabilizes, we trivially have the upper bound for the initial monotone chain, so assume we have a countable monotone in the extension order sequence $(\nu_i)$ of IRO's. By definition of the extension order, for every pair $(\nu_i, \nu_{i+1})$ there is an invariant measure $\xi_i$ on $\pord(G) \times \pord(G)$ with marginals $\nu_i$ and $\nu_{i+1}$ that is supported on the pairs $(\prec_1, \prec_2)$ such that $\prec_1 \subset \prec_2$ (we consider relations as subsets of $G \times G$). Now, repeatedly using the relatively-independent joining construction (see e.g. \cite[Example, p. 126]{Gl03}), we construct a countable joining, a $G$-invariant measure $\mu$ on $(\pord(G))^\N$ such that its marginals $\proj_i(\mu)$ are $\nu_i$ and that its pair-marginals $\proj_{i,i+1}(\mu) = \xi_i$. We conclude that $\mu$ almost-every sequence $(\prec_i) \in (\pord(G))^\N$ is monotone: $\prec_i \subset \prec_{i+1}$ for all $i \in \N$. Define a Borel function $u : (\pord(G))^\N \to 2 ^ {G \times G}$ that sends a sequence of orders to their union. Obviously, function $u$ applied to a monotone sequence results in a subset that is an order. It is now easy to check that $u(\mu)$ is a measure on $\pord(G)$ that is an upper bound in the extension order for all IRO's $\nu_i$.
\end{proof}

A similar proposition with a similar proof holds for IRO's on equivalence relations.

\begin{quest}
What can we say about the space of maximal invariant random orders on a group? What can be said about the L-action of the group on this set?
\end{quest}

{\em Acknowledgements.} 
I would like to thank the referee for useful remarks and suggestions.
The research was carried out in Euler Mathematical Institute at St. Petersburg State University, supported by Ministry of Science and Higher Education of the Russian Federation, agreement no. 075--15--2022--287.

\section{Preliminaries}

For a standard Borel space $X$ we denote $\meas(X)$ the space of all Borel probability measures.

By a pmp (probability measure preserving) action of a countable group $G$ on a standard probability space $(X, \mu)$ (or simply, a $G$-action) we mean a left action of $G$ by Borel automorphisms that preserve measure: $\mu(g^{-1}(A)) = \mu(A)$ for all measurable subsets $A$ of $X$ and all $g \in G$. We will denote $G \acts (X,\mu)$ the action described. For a topological action $G \acts X$, we denote $\invmeas(G \acts X)$ the space of all invariant Borel probability measures on $X$, or simply $\invmeas(X)$, if the action is implied, in particular $\invmeas(\pord(G))$ is the space of all IRO's on $G$.
Two pmp actions of $G$ on standard probability spaces are isomorphic iff there is a measure-preserving isomorphism that intertwines these actions.

By a joining of two pmp $G$-actions on $(X_1, \mu_1)$ and $(X_2, \mu_2)$ we mean a $G$-invariant measure $\nu$ on $X_1 \times X_2$ such that $\proj_{X_1} (\nu) = \mu_1$ and $\proj_{X_2}(\nu) = \mu_2$, where $\proj_{X_i}$ denotes the marginal-projection. Similarly, we can define a triple joining.
If $G \acts (X_i, \mu)$, $i = 1,2$ are pmp actions, we denote $\join(G \acts (X_1, \mu),\, G \acts (X_2, \mu))$ the space of all joinings of these two pmp actions. If $G \acts (X, \mu)$ is a pmp action and $G \acts Y$ is a topological action on a compact space, we denote $\join(G \acts (X, '\mu), G \acts Y)$ the space of all joinings of  $G \acts (X, \mu)$ with action $G \acts Y$ endowed with some Borel probability measure.

\section{The free group has no IRO-extension property}\label{sec: free}

In this section we will prove that a free group that factors onto $\gpSL_3(\Z)$ has no IRO-extension property. 
It would be nice to prove that the IRO-extension property passes to factor-groups, but I don't know how to do this. Instead, we manually lift the example presented in \cite{GLM22} to the factorized free group, tracing the proof from that paper. 

In this section we will work with the L-action on $\pord$ and will borrow heavily notation and statements from \cite{GLM22}. 

Let $G$ be a group and $\sqsubset$ be an order on $G$. 
For $a,b \in G$ denote:
\begin{equation*}
\begin{array}{cc}
\smlp(a,b) =&\left\lbrace \prec \in \ext(\sqsubset) \vert \,\exists q > 0 \, \forall n > 0 \quad a^{-q}b^n \prec e\right\rbrace \\
\smlm(a,b) =&\left\lbrace \prec \in \ext(\sqsubset) \vert \,\exists q > 0 \, \forall n > 0 \quad e \prec b^{-n}a^q \right\rbrace
\end{array}
\end{equation*} 

In what follows $\Gamma = \gpSL_3(\Z)$. Let $\sqsubset$ be the left-invariant order on $\Gamma$ that is generated by the semigroup consisting of matrices in the form $I + A$, where $I$ is the identity operator matrix and $A$ is a matrix with all non-negative entries and at least one strictly positive entry.
The following elements are positive in that order:

\begin{equation*}
\begin{array}{ccc}
a_1 = \begin{pmatrix}
	1 & 1 & 0 \\
	0 & 1 & 0 \\
	0 & 0 & 1
\end{pmatrix}&
a_2= \begin{pmatrix}
	1 & 0 & 1 \\
	0 & 1 & 0 \\
	0 & 0 & 1 
\end{pmatrix}& 
a_3 = \begin{pmatrix}
	1 & 1 & 0 \\
	0 & 1 & 0 \\
	0 & 0 & 1 
\end{pmatrix} \\
a_4 = \begin{pmatrix}
	1 & 0 & 0 \\
	1 & 1 & 0 \\
	0 & 0 & 1
\end{pmatrix}& 
a_5 = \begin{pmatrix}
	1 & 0 & 0 \\
	0 & 1 & 0 \\
	1 & 0 & 1 
\end{pmatrix}& 
a_6 = \begin{pmatrix}
	1 & 0 & 0 \\
	0 & 1 & 0 \\
	0 & 1 & 1 
\end{pmatrix}.
\end{array}
\end{equation*}
Note that $a_i$ and $a_{i+1}$ commute for $i = 1, \ldots, 6 \mod 6$ and generate a copy of $\Z^2$.

Denote $\smlm = \bigcap_{i = 1}^6 \smlm(a_i, a_{i-1})$ and $\smlp = \bigcap_{i = 1}^6 \smlp(a_i, a_{i+1})$.
\begin{lem}[\cite{GLM22}, Lemma 3.5]\label{lem: free_sl_extension_decomposition}
\[
\ext(\sqsubset) = \smlp \cup \smlm.
\]
\end{lem}

If $\prec$ is an order on a set $X$, and $\varphi : Y \to X$ is a map, we denote $\varphi^{Ord}(\prec)$ the order on $Y$ defined by $y_1 (\varphi^{Ord}(\prec)) y_2$ iff $\varphi(y_1) \prec \varphi(y_2)$, for every $y_1, y_2 \in Y$.

Let $F$ be a free group and let $\pi: F \to \Gamma$ be an epimorphism. A {\em transversal} is any map $\varphi$ from $\Gamma$ to $F$ such that $\pi \circ \phi$ is the identity map on $\Gamma$. We will need a transversal of a special kind. Let us fix any  $\repa_1, \ldots , \repa_6 \in F$ such that $\pi(\repa_i) = a_i$.
Note that elements $a_i$ and $a_{i+1}$ commute and generate a copy of $\Z^2$. We also have intersections $\langle a_i, a_{i+1} \rangle \cap \langle a_{i+1}, a_{i+2}\rangle = \langle a_{i+1}\rangle$, for $i = 1, \ldots, 6 \mod 6$, and $\langle a_i, a_{i+1}\rangle \cap \langle a_j, a_{j+1}\rangle = e$ if $i-j \notin \{-1, 0, 1\} \mod 6$.  
Thus we may define $\varphi(a_i^n a_{i+1}^m) = \repa_i^n \repa_{i+1}^m$, for $i = 1, \ldots, 6 \mod 6$, and $n,m \in \Z$ (we will refer to this as a {\em pseudo-homomorphism property}); we define $\varphi$ on remaining elements of $\Gamma$ arbitrarily in such a way that it becomes a transversal. 
Denote $\bsqsub = \pi^{Ord}(\sqsubset)$.
We will need the following technical lemma:

\begin{lem}\label{lem: free_lifting_order_properties}
Let $\prec \in \ext(\bsqsub)$. The following hold:
\begin{enumerate}
\item $\varphi^{Ord}(\prec) \in \ext(\sqsubset)$;
\item $\prec \in \smlbm(\repa_i, \repa_{i-1})$ iff $\varphi^{Ord}(\prec) \in \smlm(a_i, a_{i-1})$, for $i=1, \ldots, 6 \mod 6$;
\item $\prec \in \smlbp(\repa_i, \repa_{i+1})$ iff $\varphi^{Ord}(\prec) \in \smlp(a_i, a_{i+1})$, for $i=1, \ldots, 6 \mod 6$.
\end{enumerate}
\end{lem}
\begin{proof}
First note that the map $\varphi$ is a bijection between $\Gamma$ and the image of $\varphi$ such that $\pi \circ \varphi = Id$. Also note that  $\varphi$ defines an isomorphism between the restriction of $\bsqsub$ to the image of $\varphi$ and $\sqsubset$, i.e. $\varphi^{Ord}(\bsqsub) = \sqsubset$.

For (1) we use that the restiction of $\prec$ to the image of $\varphi$ is a total order, that extends $\bsqsub$, so $\varphi^{Ord}(\prec)$ is trivially a total order on $\Gamma$ that extends $\sqsubset$.

For (2) we again use that $\varphi$ is a bijection, but we also need its pseudo-homomorphism property we get by construction:
\begin{equation}\label{eq: local homomorphism}
\varphi(a_i^n a_{i+1}^m) = \alpha_i^n \alpha_{i+1}^m
\end{equation}
and 
\begin{equation*}
\pi(\alpha_i^n \alpha_{i+1}^m) = a_i^n a_{i+1}^m, 
\end{equation*}
for all $i = 1, \ldots, 6 \mod 6$ and all $n,m \in \Z$.
The statement $$\prec \in \smlbm(\alpha_i, \alpha_{i-1}),$$ by definition, means that $\prec$ is a linear extension of order $\bsqsub$ on $\Gamma$ and there is $q > 0$ such that for every $n>0$ we have $e \prec \alpha_{i-1}^{-n}\alpha_i^q$.  The latter is equivalent to $\varphi(e) \prec  \varphi(a_{i-1})^{-n}\varphi(a_i)^q$. Now we use the pseudo-homomorphism property \ref{eq: local homomorphism} to see that the latter is equivalent to  $\varphi(e) \prec  \varphi(a_{i-1}^{-n} a_i^q)$, which is equivalent, by definition, to $e (\varphi^{Ord}(\prec))  a_{i-1}^{-n} a_i^q$, so we get the needed equivalence. It is crucial that the order of multiplication is right so we could use the pseudo-homomorphism property.

The last point, (3), is proved similarly, we only need to  carefully observe that the pseudo-homorphism property is applicable as well.
\end{proof}

Denote 
$\bsmlm = \bigcap_{i=1}^6 \smlbm(\repa_i, \repa_{i-1})$ and 
$\bsmlp = \bigcap_{i=1}^6 \smlbp(\repa_i, \repa_{i+1})$.
\begin{lem}\label{lem: free extension decomposition}
\[
\ext(\bsqsub) = \bsmlm \cup \bsmlp. 
\]
\end{lem}
\begin{proof}
It is trivial that the right-hand side is contained in the left-hand side. For the other direction, let $\prec \in \ext(\bsqsub)$. By the previous lemma this implies that $\varphi^{Ord}(\prec) \in \ext(\sqsubset)$, so by Lemma \ref{lem: free_sl_extension_decomposition}, $\varphi^{Ord}(\prec) \in \smlm \cup \smlp$. Using previous lemma again, we get that $\prec \in \bsmlm \cup \bsmlp$.
\end{proof}

Let the group $F$ be acting by Borel automorphisms on a standard Borel space $X$. A Borel subset $A$ of $X$ is called {\em wandering} if there is a $g \in F$ such that sets $g^n A$ are pairwise disjoint for $n \in \Z$. Denote $\wand$ the collection of all countable unions of wandering sets.
\begin{lem}[\cite{GLM22}, Lemma 3.6]\label{lem:free wandering properties}
Class $\wand$ has the following properties.
\begin{enumerate}
\item $\wand$ is closed under taking countable unions and measurable subsets;
\item any set from $\wand$ does not support an $F$-invariant measure.
\end{enumerate}
\end{lem}

\begin{lem}[\cite{GLM22}, Lemma 3.7]\label{lem:free wandering criterion}
Let $A$ be a measurable subset of $X$. Let $g_1, \ldots , g_k$ be a $k$-tuple of elements from $F$. If there is $n > 0$ such that for any $k$-tuple $n_1, \ldots, n_k > n$ we have $\bigcap_{i=1}^k g_i^{n_i} A = \varnothing$, then $A \in \wand$.
\end{lem}

\begin{lem}\label{lem: free bsmlp in wand}
$\bsmlp \in \wand$.
\end{lem}
\begin{proof}
Let $\bar{A}(q, i) = \lbrace \prec \in \ext(\bsqsub) \vert\, \repa_i^{-q} \repa_{i+1}^n \prec e \, \forall n>0\rbrace$, for $i = 1, \ldots, 6 \mod 6$ and $q > 0$. Note that $\bar{A}(q, i) \subset \bar{A}(q',i)$, for $q < q'$ since 
\begin{equation*}
\pi(\repa_i^{-q'} \repa_{i+1}^n) = a_i^{-q'} a_{i+1}^n \sqsubset a_i^{-q}a_{i+1}^n = \pi(\repa_i^{-q} \repa_{i+1}^n).
\end{equation*}
The latter implies that 
\begin{equation*}
\bsmlp = \bigcup_{q > 0} \bigcap_{i = 1}^6 \bar{A}(q,i).
\end{equation*}
It is enough to show by Lemma \ref{lem:free wandering properties} that $\bigcap_{i = 1}^6 \bar{A}(q,i) \in \wand$. For the latter we will use Lemma \ref{lem:free wandering criterion}. We will prove that 
for $n_1, \ldots, n_6 > q + 1$ we have $\bigcap_{i = 1}^6 \repa_{i+1}^{-n_i}\bar{A}(q,i) = \varnothing$. Let $n > q+1$ and $\prec \in \repa_{i+1}^{-n}\bar{A}(q,i)$. We get $\repa_{i+1}^{-n} \repa_i^{-q} \repa_{i+1}^n \prec \repa_{i+1}^{-n}$, so
\[
	\repa_i^{-(q+1)} \prec \repa_{i+1}^{-n} \repa_i^{-q} \repa_{i+1}^n \prec \repa_{i+1}^{-n} \prec \repa_{i+1}^{-(q+1)},
\]
in the first inequality we used that 
\[\pi(\repa_i^{-(q+1)}) = a_i^{-(q+1)} \sqsubset a_i^{-q} = \pi(\repa_i^{-q}) = \pi(\repa_{i+1}^{-n} \repa_i^{-q} \repa_{i+1}^n)
\] 
and $\prec \in \ext(\sqsubset)$. So for any $\prec \in \bigcap_{i = 1}^6 \repa_{i+1}^{-n_i}\bar{A}(q,i)$ with $n_1, \ldots, n_6 > q + 1$ we get 
\[
\repa_1^{-(q+1)} \prec \repa_2^{-(q+1)} \prec \ldots \repa_6^{-(q+1)} \prec \repa_1^{-(q+1)},
\]
a contradiction.
\end{proof}

\begin{lem}\label{lem: free bsmlm in wand}
$\bsmlm \in \wand$.
\end{lem}

\begin{proof}
The proof is analogous to the previous but for minor details.

Denote $\bar{B}(q, i) = \lbrace \prec \in \ext(\bsqsub)\vert\, e \prec \repa_{i-1}^{-n} \repa_i^q \quad \forall n > 0\rbrace$. As in the previous proof, it is enough to show that $\bigcap_{i=1}^6\repa_{i-1}^{n_i} \bar{B}(q,i) = \varnothing$, for $n_1, \ldots, n_6 > q$. Consider $\prec \in \repa_{i-1}^n \bar{B}(q,i)$. Observe that $\repa_{i-1}^n \prec \repa_i^q$, for $n > 0$. So for $\prec \in \bigcap_{i=1}^6\repa_{i-1}^{n_i} \bar{B}(q,i)$ with $n_1, \ldots, n_6 > q$ we obtain 
\[
\repa_1^q \prec \repa_2^q \prec \ldots \prec \repa_6^q \prec \repa_1^q,
\]
a contradiction.
\end{proof}

We are ready to prove Theorem \ref{thm: free no iro}. stating that the free group $F_2$ (in fact, any group that factors onto $\gpSL_3(\Z)$) has no IRO-extension property.

\begin{proof}
Let $F$ be the free group on two generators. 
Note that $\ext(\bsqsub) = \bsmlp \cup \bsmlm$ (by Lemma \ref{lem: free extension decomposition}) so $\ext(\bsqsub) \in \wand$ (by Lemmata \ref{lem: free bsmlm in wand}, \ref{lem: free bsmlp in wand} and Lemma \ref{lem:free wandering properties}(1)).  
Consider the invariant random partial order given by the $\delta$-measure on $\bsqsub$. 
Note that an extension of the latter random order will give us an invariant measure on $\ext(\bsqsub)$ which is impossible by Lemma \ref{lem:free wandering properties}(2).
\end{proof}

\section{Preliminaries on ergodic theory of countable measurable equivalence relations}

Before proceeding with the proof of the main result we need to remind some standard facts and notions from ergodic theory, especially regarding equivalence relations, measure decomposition and invariance of measures.

A {\em measurable space} is a set endowed with a $\sigma$-algebra. A map between measurable spaces is {\em Borel} if preimages of all measurable sets are measurable. An isomorphism of measurable spaces is bijection that is Borel together with its inverse. 
A {\em Polish space} is a topological space whose topology comes from a complete separable metric structure. 
A {\em standard Borel space} is a measurable space that is measurably isomorphic to a Polish space endowed with the Borel $\sigma$-algebra. We refer the reader to \cite{Ke95} for the introduction to the subject.
Standard Borel spaces are completely classified up to isomorphism by their cardinality: there are all finite, a countable and the unique space of continuum cardinality, see \cite[Theorem 15.6, p. 90]{Ke95}. 
A Borel subset of a standard Borel space endowed with the induced $\sigma$-algebra is a standard Borel space itself \cite[Corolary 13.4, p. 82]{Ke95}.
From now on all measurable spaces are standard Borel.

An image of a Borel set under a Borel map is not always a Borel set (see \cite[Theorem 14.2, p. 85]{Ke95}),
On the other hand,

\begin{lem}\cite[Corollary 15.2 p. 89]{Ke95}\label{lem: injective Borel}
Let $f$ be a Borel map between standard Borel spaces $X$ and $Y$. If $f$ is injective on a Borel subset $A$ of $X$ then $f(A)$ is a Borel subset of $Y$.
\end{lem}

\begin{lem}\cite[Theorem 14.12, p.88]{Ke95}\label{lem: borel map graph}
Let $X$, $Y$ be standard Borel spaces. A partial function from $X$ to $Y$ has a Borel graph iff it is a Borel function and its domain is a Borel subset of $X$.
\end{lem}

\begin{lem}[Lusin-Novikov uniformization, see Theorem 18.10 p. 123 in \cite{Ke95}]\label{lem: lusin-novikov}
Let $X$, $Y$ be standard Borel spaces and let $P \subset X \times Y$ be Borel. If $(\{x\} \times Y) \cap P$ is at most countable for every $x \in X$ then $P$ is a union of an at most countable collection of Borel sets $(P_i)$, $i \in \N$ such that the fiber$(\{x\} \times Y )\cap P_i$ has at most one point for every $x \in X$ and $i \in \N$.
\end{lem}

\begin{cor}
Let $X$, $Y$ be standard Borel spaces and let $A$ be a Borel subset of $X$ such that section $\{x\} \times Y \cap A$ are at most countable for all $x \in X$, then the map $x \mapsto \lvert( {x}\times Y) \cap A\rvert$ is a Borel map from $X$ to $\N \cup \{\infty\}$.
\end{cor}

\begin{cor}\label{cor: strict uniformization}
Let $X$, $Y$ be standard Borel spaces and let $A \subset X \times Y$ be a Borel subset such that $(\{x\} \times Y) \cap A$ is countable for every $x \in X$. Then there is a collection of Borel maps $t_i : X \to Y$, $i \in \N$ such that for every $x \in X$ and $i, j$ from $\N$ with $i \neq j$ we have $t_i(x) \neq t_j(x)$, and $(\{x\} \times Y )\cap A = \{(x, t_i(x)) \vert i \in \N\}$ for every $x \in X$.
\end{cor}

\begin{cor}\label{cor: countable to one}
Let $f: X \to Y$ be a Borel map between standard Borel spaces. If $A$ is such a Borel subset of $X$ that $f^{-1}(y) \cap A$ is at most countable for all $y \in Y$ then $f(A)$ is a Borel subset of $Y$.
\end{cor}

For an equivalence relation $E$ on a set $X$ we denote $[x]_E$ the $E$-equivalence class of a point $x \in X$. We say that a function on $X$ is $E$-invariant if it is constant on equivalence classes. We say that a subset of $X$ is $E$-invariant if the indicator function of the subset is invariant.
\begin{deff}
Let $X$ be a standard Borel space. A {\em countable Borel equivalence relation} on $X$ is a Borel subset $E$ of $X \times X$ that is an equivalence relation and such that equivalence classes $[x]_E$ are countable for all $x \in X$.
\end{deff}

\begin{lem}\label{lem: nested unifimization}
Let $E_1$, $E_2$ be two Borel countable equivalence relations on a standard Borel space $X$ such that $E_1 \subset E_2$. Let $\upsilon : X \to \N \cup \{\infty\}$ be a function defined by $\upsilon(x) = \lvert [x]_{E_2} / E_1\rvert$, that is for each $x\in X$ it counts the number of $E_1$-equivalence classes the $E_2$-equivalence class of $x$ splits into. This function is $E_2$-invariant.  There is a collection of partial Borel maps $(t_i)_{i \in \N}$ such that 
\begin{enumerate}
\item $t_i(x)$ is defined iff $i \in [0, \upsilon(x))$, for $x \in X$;
\item  for all $x$ holds $x E_2 t_i(x)$ whenever  $t_i(x)$ is defined;
\item for all $x, y \in X$ such that $x E_2 y$ there is $0 \leq i < \upsilon(x)$ such that $t_i(x) E_1 y$;
\item $t_i(x) \neq t_j(x)$ for all $x \in X$ and $i \neq j$.
\end{enumerate}
\end{lem}
\begin{proof}
For two subsets $A$ and $B$ of $X \times X$ we define $A \circ B$ to be the set of all $(x, z) \in X \times X$ such that there is $y \in X$ with $(x,y) \in A$ and $(y, z) \in B$. Note that if $A$ is a Borel set and $E$ is a countable Borel equivalence relation then $A \circ E$ is a Borel set by Corolary \ref{cor: countable to one}.
We present an inductive construction of functions $t_i$ that simultaneously shows that function $\upsilon$ is Borel.
We define $t_0(x) = x$ for all $x \in X$. 
We apply the Lusin-Novikov uniformization theorem (Lemma \ref{lem: lusin-novikov}) to set $E_2$. 
We will use capitalized letters to denote graphs of functions and lowercase letters to refer to functions themselves, in particular, $p_i$ will denote the partial Borel functions those graphs are sets $P_i$ from the Lusin-Novikov theorem.

Assume that $t_{i-1}$ is already constructed. 
For each $x$ we find the smallest $j$ such that $p_j(x)$ is defined and is not $E_1$-equivalent to $t_k(x)$, for all $k<i$, and then we define $t_i(x) = p_j(x)$, leaving $t_i(x)$ undefined if there are no $j$ found. From this construction and the definition of sets $P_i$ from the Lusin-Novikov uniformization theorerm follows that functions $t_i$ satisfy all the requirement, except maybe that of being Borel. To see that they are Borel we observe that their graphs are Borel sets (they are constructed by means of countable unions, intersection, as well as projections of sets with at most countable preimages) and use Lemma \ref{lem: borel map graph}.

In more details. We assume that $t_k$ is already constructed for all $k < i$. We define inductively a sequence of partial maps $(t_i^j)_{j \in \N}$ in the following way. 
Let 
\[
T_i^j = P_j \setminus \left(\bigcup_{k<i} T_k \circ E_1 \cup \left(\proj_1\bigcup_{r<j} T_i^r \times X\right)\right),
\]
$\proj_1$ is the natural projection map from $X \times X$ to the first coordinate. We remind that $T_i^j$ denotes the graph of the partial function $t_i^j$.  Note that all sets $T_i^j$ are Borel by the standard properties of the Borel $\sigma$-algebra and by Lemma \ref{lem: injective Borel}.
We define
\[
T_i = \bigcup_{j \in \N} T_i^j.
\]
We defined $t_i^j(x)$ to be  $p_j(x)$, unless $p_j(x)$ is $E_1$-equivalent to  $t_k(x)$ for $k<i$ or $t_i^r(x)$ is already defined for $r < j$. It is easy to see that $t_i(x)$ is undefined iff for every $y E_2 x$ we have $y E_1 t_k(x)$ for some $k < i$. By Lemma \ref{lem: borel map graph}, maps $t_i$ (with graphs $T_i$) are partially-defined Borel maps. We also note that $\upsilon(x) = i$ iff $t_{i-1}(x)$ is defined and $t_i(x)$ is undefined, so the set of such points is precisely $\proj_1(T_{i-1} \setminus T_i)$.
\end{proof}
 
\begin{deff}
A countable Borel equivalence relation $E$ on a standard probability space $X$ is {\em measure preserving} (pmp) if for every Borel bijection $\psi$ between Borel subsets $A$ and $B$ of $X$, whose graph is a subset of $E$, we have $\mu(A) = \mu(B)$. We also would say that $\mu$ is $E$-invariant.
\end{deff}

The main source of countable pmp Borel equivalence relations is probability measure preserving actions of countable groups. Consider an action $G \curvearrowright X$ of a countable group $G$ on a standard probability space $(X, \mu)$ by measure preserving Borel automorphisms. The orbit equivalence relation of that action is an equivalence relation given by $x \sim y$ if $y = gx$ for some $g \in G$. It is indeed a pmp Borel countable equivalence relation on $(X, \mu)$ by the following lemma.

\begin{lem}\cite[see Theorem 1, Remark 1, Corollary 1]{FM77}\label{lem: equiv invariance reformulations}
Let $E$ be a countable Borel equivalence relation on a standard probability space $(X, \mu)$
The following are equivalent.
\begin{enumerate}
\item $E$ is pmp;
\item every partial Borel bijection $T$ whose graph is a subset of $E$ is measure-preserving;
\item there is a finite or countable collection of partial Borel measure-preserving bijections $(T_i)$, union of whose graphs is $E$;
\item there is a finite or countable collection of partial Borel measure-preserving bijections $(T_i)$, whose graphs are disjoint and whose union is $E$;
\item for any subset $Q$ of $E$ we have $$\int_X  \Bigl\lvert (\{x\} \times X) \cap Q \Bigr\rvert d\mu(x)  = \int_X \Bigl\lvert (X \times \{x\}) \cap Q \Bigr\rvert d\mu(x);$$ 
\item there is a Borel pmp action of a countable group on $X$ such that $E$ is the orbit equivalence relation of that action. 
\end{enumerate}
\end{lem}

We also have the following:

\begin{lem}\cite[see Corollary 1]{FM77}\label{lem: G E invariant equiv}
Let $G$ be a countable group that acts by Borel automorphisms on a standard Borel space $X$. 
Let $E$ be the orbit equivalence relation of that action. A Borel probability measure on $X$ is $G$ invariant iff it is $E$-invariant. 
\end{lem}

\begin{deff}
An equivalence relation is {\em ergodic} if every invariant measurable subset (i.e. a set that contains together with each point its equivalence class) has measure $0$ or $1$.
\end{deff}

\begin{lem}\label{lem: measure decomposition}[Theorem 5.14 in \cite{EW11}, p. 135]
Let $X,Y$ be standard Borel spaces, let $p: Y \to X$ be a Borel surjection. Let $\mu$ be a Borel probability measure on $X$, and $\nu$ be a Borel probability measure on $Y$ such that $p(\nu) = \mu$.  There is a unique(up to a set of zero measure) measurable map $X \to \meas(Y)$, $x \mapsto \nu_x^p$ such that $\nu^p_x \in \meas(p^{-1}(\{x\}))$ and
$\expect(f \vert X)(x) = \int_Y f(y)d\nu_x^p(y)$ for $\mu$-a.e. $x \in X$. Also, $\nu = \int_X \nu^p_x d\mu(x)$.
\end{lem}

\begin{deff}
An pmp countable Borel equivalence relation $E$ on a standard probability space $(X, \mu)$ is called {\em amenable} if there is a sequence of measure-valued functions $(\tau_i)$ such that $\tau_i(x) \in \meas([x]_E)$, for all $x \in X$ and $i \in \N$, and for $\mu$-almost every $x \in X$ and every $y \in [x]_E$ we have 
\[
\lim_{i \to \infty}\lvert \tau_i(x) - \tau_i(y)\rvert = 0,
\]
$\lvert \cdot \rvert$ stands for the total variation norm.
\end{deff}

\begin{deff}
A countable Borel equivalence relation $E$ on a standard Borel space $X$ is called {\em hyperfinite} if it is a union of an increasing sequence of finite equivalence relations. Assuming that $X$ is endowed with a Borel probability measure $\mu$, we say that $E$ is hyperfinite almost everywhere, if its restriction to a full-measure subset of $X$ is hyperfinite.
\end{deff}

\begin{deff}
A group $G$ is called amenable if for every finite subset $S$ of $G$ and every $\varepsilon > 0$  there is a non-empty finite subset $F= F(S, \varepsilon)$ of $G$ such that 
\[
\lvert SF \setminus F\rvert \leq \varepsilon \lvert F\rvert.
\]
\end{deff}

We collect important facts about amenability in the context of orbit equivalence in the following theorem.
\begin{theor}[\cite{Dy59}, \cite{CFW81}, \cite{OW80}, \cite{Ka97} and \cite{KeM04}, Theorems 10,1 and 10.7]\label{thm: amenability}
A pmp countable equivalence relation on a standard probability space is amenable iff it is hyperfinite a.e. The orbit equivalence relation of a pmp action of a countable amenable group on a standard probability space is amenable. All ergodic countable pmp amenable equivalence relations on standard probability spaces are isomorphic.
\end{theor}

\begin{deff}
We say that a standard Borel space $Z$ is {\em fibered} over a standard Borel space $X$ if a Borel map $p: Z \to X$ is fixed. The {\em fiber} $Z_x$ of $Z$ over $ \in X$ is defined by $Z_x = p^{-1}$(\{x\}).
A {\em Borel action of a countable Borel equivalence relation} $(X,E)$ (or an {\em extension of the equivalence relation}) is a fibered over $X$ space $Z$ endowed with a countable Borel equivalence relation $E_Z$ such that $p: Z \to X$ is {\em class-bijective}, that is $p$ is a bijection between equivalence classes $[z]_{E_Z} \to [p(z)]_{E}$ for all $z$ in $Z$. The latter implies that for all $x,y\in X$ such that $x E y$, we may define a Borel map $\lift_{X,Z}(x,y) : Z_x \to Z_y$ by $\lift_{X,Z}(x,y)(z) = z'$ for $z \in Z_x$ iff $z' \in Z_y$ and $z E_Z z'$. 
\end{deff}

We would usually prefer the term ``action'' instead of ``extension'' so as to not make confusion with order-extensions.

\begin{deff}
Let $Z$ be a fibered Borel space over a standard probability space $(X, \mu)$. A measure-section of $Z$ is a map $f: X \to \meas(Z)$ such that $f(x) \in \meas(p^{-1}(\{x\}))$. We say that two measure-sections $f_1, f_2$ are equivalent if $f_1(x) = f_2(x)$ for $\mu$-a.e. $x \in X$.  
\end{deff}

\begin{deff}
Let $(Z, E_Z, p)$ be a Borel action action of $(X, E)$. Let $X$ be endowed with a Borel probability measure $\mu$. An {\em invariant measure-section} 
for this action is a map $x \to \meas(Z)$, $x \mapsto \eta_x$ such that for almost every $x\in X$ holds $\eta_x(Z_x) = 1$ and for all $y \in [x]_E$ we have $\lift_{X,Z}(x,y)(\eta_x) = \eta_y$.
\end{deff}

\begin{lem}\label{lem: inv measure is invariant section}
The set of classes of invariant measure-sections (equal on full $\mu$-measure subsets of $X$) on $Z$ is naturally isomorphic to the set of all $E_Z$-invariant measures on $Z$ whose projection on $X$ is $\mu$. Namely, for any invariant measure-section $\eta_x$ on $Z$ over $(X, \mu)$, the integral $\kappa$ is a $E_Z$-invariant measure on $Z$ whose projection is on $X$ is $\mu$. For any $E_Z$-invariant measure $\kappa$ on $Z$, there is a unique invariant measure-section on $Z$, $\eta_x$, given by the measure decomposition of $\kappa$ over $\mu$, such that $\kappa= \int_X \eta_x d\mu(x)$.
\end{lem}

For an extension $(Z, E_Z)$ of a measure-preserving equivalence relation $(X, \mu, E)$, we denote $\invmeas(Z, \mu, E_Z)$ the space of al $E_Z$ - invariant Borel probability measures on $Z$ those projection on $X$ is $\mu$ (the equivalence relation $(X, \mu, E)$ would usually be implied).

Let $(Z_1, E_1)$ and $(Z_2, E_2)$ be two Borel actions of a countable Borel equivalence relation $E$ on a standard Borel space $X$. We say that that $Z_2$ is a factor of $Z_1$ if a Borel map $\varphi: Z_1 \to Z_2$ is fixed such that $p_1 = p_2 \circ \varphi$, where $p_i : Z_i \to X$ are the natural projections. 
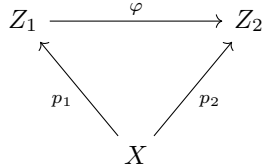
\[
\begin{tikzcd}
Z_1 \ar[rr, "\varphi"]\ar[ddr, "p_1"']& & Z_2\ar[ddl, "p_2"] \\\\
 & X 
\end{tikzcd}
\]

From the previous lemma we get the following:
\begin{lem}\label{lem: invariant measure extension property}
Assume that $(Z_1, E_1)$ and $(Z_2, E_2)$ are two actions (class-bijective extension) of a countable pmp Borel equivalence relation $E$ on a standard probability space $(X, \mu)$  such that $(Z_2, E_2)$ is a factor of $(Z_1, E_1)$.
The following are equivalent:
\begin{enumerate}
\item for every $E_1$-invariant measure $\kappa$ on $Z_2$ such that $p_2(\kappa) = \mu$ there is an invariant measure $\gamma$ on $Z_1$ such that $\varphi(\gamma) = \kappa$;
\item for every invariant measure-section $(\eta_x)_{x \in X}$ on $(Z_2, E_2)$ there is an invariant measure-section $(\xi_x)_{x \in X}$,  such that $\varphi(\xi_x) = \eta_x$ for $\mu$-a.e. $x \in X$.
\end{enumerate}
\end{lem}

\begin{deff}
In the setting of the previous lemma we would say that the factorization $(Z_1, E_1) \to (Z_2, E_2)$ over the pmp Borel equivalence relation $(X, \mu, E)$ has {\em the invariant measure extension property}, if the equivalent requirements hold.
\end{deff}

For any two actions $(Z_1, E_1)$ and $(Z_2, E_2)$ of a countable Borel equivalence relation $(E, X)$, we define their product $(Z_1, E_1) \otimes_{(X,E)} (Z_2, E_2)$ as a set $\{(z_1, z_2) \vert z_1 \in Z_1, z_2 \in Z_2, p_1(z_1) = p_2(z_2)\}$. We endow it with equivalence relation: $(z_1, z_2) \sim (z'_1, z'_2)$ iff $z_1 \sim z'_1$ and $z_2 \sim z'_2$. The latter makes $(Z_1, E_1) \otimes_{(X,E)} (Z_2, E_2)$ a class-bijective extension of $(X,E)$.
Assume $X$ is endowed with an $E$-invariant probability measure $\mu$, $Z_1$ and $Z_2$ - with invariant measures $\nu_1, \nu_2$ respectively such that $p_1(\nu_1) = \mu$ and $\pi_2(\mu_2) = \mu$. We can endow $Z_1 \otimes_{(X,E)} Z_2$ with the so-called {\em relatively-independent product} measure defined by $\lambda = \int_X \nu_1\vert_x^{p_1} \otimes \nu_2\vert_x^{p_2}$. It is easy to check that $\lambda$ is an invariant measure on  $Z_1 \otimes_{(X,E)} Z_2$, using Lemma \ref{lem: inv measure is invariant section}.

\section{Orders on equivalence relations}\label{sec: orders on equiv}

We remind that if $\varphi: A \to B$ is a map, then we can define $\varphi^{Ord}: \pord(B) \to \pord(A)$ by $a_1 (\varphi^{Ord}\prec) a_2$ iff $\varphi(a_1) \prec \varphi(a_2)$. We also can push measures forward along $\varphi^{Ord}$.

Let $E$ be a Borel equivalence relation on a standard Borel space $X$. We define the order-action $\pord(X,E)$ of $E$ in the following way. The $X$-fibered space would be the space of all pairs $\{(x, \omega)\vert\, x\in X, \omega \in \pord([x]_E)\}$. The latter space does not have an {\em a priori} standard Borel structure. To endow it with one, we note that it could be identified with the space $X \times \pord(\N)$, the identification is given by a full injective set of Borel transversals $(t_i)_{i \in \N}$ i.e. Borel maps $t_i: X \to X$ such that $[x]_E = \{t_i(x)\}$ and  $t_i(x) \neq t_j(x)$ for $i \neq j$, see Corollary \ref{cor: strict uniformization}. We note that that the Borel structure obtained does not depend on the choice of the collection of Borel transversals.

The equivalence relation on $\pord(X,E)$ is given by $(x_1, \omega_1)$ is equivalent to $(x_2, \omega_2)$ iff $x_1 E x_2$ and $\omega_1 = \omega_2$. We also define the order-extension-action $\ordext(X,E)$ of $E$. The $X$-fibered space would be $\{(x, \omega, \omega') \vert\, x\in X, \omega \in \pord([x]_E), \omega' \in \ext(\omega) \}$, two points $(x_1, \omega_1, \omega'_1)$ and $(x_2, \omega_2, \omega'_2)$ would be considered equivalent if $x_1 E x_2$, $\omega_1 = \omega_2$ and $\omega'_1 = \omega'_2$. Note that $\pord(X,E)$ is a factor of $\ordext(X,E)$. 


\begin{deff}
Let $E$ be a pmp countable Borel equivalence relation on a standard probability space $(X,\mu)$. 
An {\em invariant random partial order} is an invariant measure-section $(\eta_x)_{x \in X}$ on $\pord(X, E)$. We say that the invariant random order can be extended if there is an invariant measure-section $(\xi_x)_{x \in X}$ on $\ordext(X, E)$ such that $\proj_{\pord(X,E)}(\xi_x) = \eta_x$ for $\mu$-a.e. $x \in X$. We say that $E$ has the {\em IRO(invariant random order)-extension property} if every invariant random partial order on $E$ can be extended as described earlier. 
\end{deff}

Let $E_1 \subset E_2$ be a countable Borel pmp equivalence relations on $(X, \mu)$.
Note that there is a natural morphism of Borel actions of $E$, $$\proj_{\pord(X,E_1)}: \pord(X, E_2) \to \pord(X, E_1).$$ Its value on $(x, \omega)$, for $x \in X$ and $\omega \in \pord([x]_{E_2})$, would be $(x, \omega \vert_{[x]_{E_1}})$. Similarly a natural morphism $$\proj_{\ordext(X, E_1)} : \ordext(X, E_2) \to \ordext(X, E_1)$$ is defined. 
\begin{lem}\label{lem: iro induction}
For every random partial order $\eta$ on $E_1$ there is an induced partial order $\eta^{E_2:E_1}$ such that $\proj_{\pord(X,E_1)}(\eta_x^{E_2:E_1}) = \eta_x$. 
\end{lem}
\begin{proof}
For every $x \in X$, the equivalence class $[x]_{E_2}$ is a disjoint union of $E_1$-equivalence classes, 
apply Lemma \ref{lem: nested unifimization}, to get representatives $t_i(x)$ of these. We define the random order on $[x]_{E_2}$ independently on each $E_1$-class $[t_i(x)]_{E_1}$ by $\eta_{t_i(x)}$, and leave elements from different $E_1$-equivalence classes uncomparable. It is easy to see that the obtained measure does not depend on the choice of representatives $t_i(x)$, by $E_1$-invariance of $\eta$, for almost every $x$, moreover, the measures would be equal for $E_2$-equivalent points.
\end{proof}

\begin{lem}\label{lem: iro equiv hereditary}
Let $E_1 \subset E_2$ be two pmp countable Borel equivalence relations on a standard probability space $(X, \mu)$. If $E_2$ has the IRO-extension property then so does $E_1$. 
\end{lem}
\begin{proof}

The following illustrates the proof, dashed lines represent lifts of measures:
\begin{equation*}
\begin{tikzcd}
\ordext(X, E_2) \ar[r] \ar[d] &\ordext(X, E_1) \ar[d]\\
\pord(X, E_2) \ar[r] \uar[dashed, bend left, "\text{assumption}"] & \pord(X, E_1) \lar[dashed, bend left, "\text{Lemma  \ref{lem: iro induction}}"]
\end{tikzcd}
\end{equation*}
Let $\eta$ be an invariant random order on $(X,E_1)$. By the previous lemma, there is an induced invariant random order $\eta' = \eta^{E_2 : E_1}$ on $E_2$. By the IRO-extension property of $E_2$, we have an extension of $\eta'$, an invariant measure-section $\xi'$ of $\ordext(X,E_2)$ such that $\proj_{\pord(X,E_2)}\xi'_x = \eta'_x$ for $\mu$-a.e. $x \in X$. Define $\xi_x = \proj_{\ordext(X,E_1)}(\xi'_x)$, for $x \in X$. It is clear that $\xi_x$ is an invariant measure section of $\ordext(X,E_1)$ and $\proj_{\pord(X,E_1)}(\xi_x) = \eta_x$ for $\mu$-a.e. $x \in X$.
\end{proof}
We define the order-extension property for equivalence relations on the level of fibers, this was instrumental for the induction lemma above. In the sequel we will need an equivalent definition in terms of lifts of ``integrated'' measures that follows simply from Lemma \ref{lem: invariant measure extension property} on the invariant measure extension property:
\begin{lem}\label{lem: iro fiberwise eq full}
A pmp countable Borel equivalence relation $E$ on $(X,\mu)$ has the IRO-extension property iff for every invariant measure $\kappa$ on $\pord(X,E)$ such that $\proj_X(\kappa) = \mu$, there is an invariant measure $\gamma$ on $\ordext(X,E)$ such that $\proj_{\pord(X,E)}(\gamma) = \kappa$. In other words, $(X, \mu, E)$ has the IRO-extension property iff the factorization of two actions $\ordext(X,E) \to \pord(X,E)$ of $(X, \mu, E)$ has the invariant measure extension property.
\end{lem}

Let $E$ be an orbit equivalence relation of a free Borel action of a countable group $G$. 
Observe that $\pord(X,E)$ is isomorphic to the orbit equivalence relation of the action $G \acts X \times \pord(G)$. 
In order to see this, notice that fibered spaces $\proj_X : X \times G \to X$ and $\proj_1 : E \to X$ ($\proj_1$ is the projection from $E \subset X \times X$ to the first coordinate) are isomorphic: $(x, g) \mapsto (x, xg)$. Note that $\lift_{X,\pord(X,E)}(x, gx) (\omega)= g\omega$, for $x \in X$, $\omega \in \pord(G)$ and $g \in G$.
In the similar way, $\ordext(X,E)$ is isomorphic to the orbit equivalence relation of the action $G \acts X \times \ordext(G)$. 

For a Borel action $G \acts X$ we denote $\orb(G \acts X)$ the orbit equivalence relation of this action (we remind that two points $x,y \in X$ are equivalent under this equivalence relation iff $x = gy$ for some $g \in G$). With this notation we get the following.
\begin{lem}\label{lem: cd for free actions}
Let $G \acts X$ be a free Borel action of a countable group on a standard Borel space and let $E$ be the orbit equivalence relation of said action. We have the following commutative diagram:
\begin{equation*}
	\begin{tikzcd}
		\ordext(X, E)\dar \ar[r,leftrightarrow] &\orb(G \acts X \times \ordext(G))\dar["id \times \proj_{\pord(G)}"]\\
		\pord(X, E) \ar[r, leftrightarrow] \dar & \orb(G \acts X \times \pord(G))\dar\\
		(X, E) \ar[r, leftrightarrow] &\orb(G \acts X)
	\end{tikzcd}
\end{equation*}
\end{lem}
This implies that the space $\invmeas(\pord(X, \mu, E))$ of all IRO's on $E$ is isomorphic to the space $\join(G \acts(X,\mu), G \acts \pord(G))$ of all joining of the action $G \acts (X, \mu)$ with the topological action $G \acts \pord(G)$ endowed with some invariant measure. 
A trivial consequence is that there is a natural map $\theta : \invmeas(G \acts \pord(G)) \to \invmeas(\pord(X, \mu, E))$ that sends an IRO $\nu \in \invmeas(\pord(G))$ to the measure $\nu \otimes \mu$ that corresponds to the product-joining. It is easy to see that this map is a one-sided inverse of the projection $\proj_{\pord(G)}$, but the latter map is by no means injective on $\invmeas(\pord(X, \mu, E)) \simeq \join (G \acts (X, \mu), G \acts \pord(G))$ (equivalently, $\theta$ is not surjective). To see the latter we remind that by \cite[Corolary 5.4]{GLM22}, any essentially free ergodic action on the standard probability space could be realized as an IRO. In particular, if we take this action to be Bernoulli of finite entropy and take $G \acts (X, \mu)$ to be Bernoulli of finite entropy we may have many non-product joinings of $G \acts (X, \mu)$ and $G \acts \pord(G)$ (assuming that $G$ is amenable or even sofic, so that the entropy distinguishes Bernoulli actions).
\begin{lem}\label{lem: iro group and equiv}
Let $E$ be the orbit equivalence relation of an essentially free pmp Borel action of a countable group $G$ on a standard probability space $(X, \mu)$. The group $G$ has the IRO-extension property iff $E$ has the IRO-extension property.
\end{lem}

\begin{proof}
Throwing away the set of measure $0$, we may assume that the $G$-action is free so that the conclusion of Lemma \ref{lem: cd for free actions} holds.
 
For one direction, assume that $E$ has the IRO-extension property. We want to prove that group $G$ has the IRO-extension property. Let $\nu$ be an invariant measure on $\pord(G)$. Consider the measure $\kappa = \mu \otimes \nu$ on $X \times \pord(G)$. By Lemma \ref{lem: G E invariant equiv}, $\kappa$ is an invariant measure for the equivalence relation on $\pord(X,E)$, so there is an invariant measure $\xi$ on $\ordext(X, E) \simeq X \times \ordext(G)$ such that $\proj_{X \times \pord(G)} \xi = \kappa$, by our assumption and Lemma \ref{lem: iro fiberwise eq full}. Note that $\xi$ is $G$-invariant, again by Lemma \ref{lem: G E invariant equiv}, and that $\proj_{\pord(G)}(\xi) = \nu$. So $\proj_{\ordext(G)}(\xi)$ is the measure on $\ordext(G)$ required by the definition of the IRO-extension property for $G$.

For the other direction. Assume that $G$ has the IRO-extension property. We want to prove that $E$ has the IRO-extension property. We will use the equivalent definition given by Lemma \ref{lem: iro fiberwise eq full}. Let $\kappa$ be an invariant measure on $\pord(X,E) \simeq X \times \pord(G)$ such that $\proj_{X} (\kappa) = \mu$. We want to prove that there is an invariant measure $\xi$ on $\ordext(X,E) \simeq X \times \ordext(G)$ such that $\proj_{\pord(X, E)}(\xi) = \kappa$. Let $\nu = \proj_{\pord(G)}(\kappa)$. By the IRO-extension property for $G$, there is an invariant measure $\lambda$ on $\ordext(G)$ such that $\proj_{\pord(G)}(\lambda) = \nu$. Note that we may consider $\lambda$ as an invariant measure on $\pord(G) \times \ord(G)$ concentrated on the closed subset $\ordext(G)$. Since $\proj_{\pord(G)}(\kappa) = \proj_{\pord(G)}(\lambda) = \nu$, there is a triple-joining measure (for example, the relatively-independent joining, see \cite[Example, p. 126]{Gl03}) $\xi$ on $X \times \pord(G) \times \ord(G)$ such that $\proj_{X \times \pord(G)}(\xi) = \kappa$ and $\proj_{\pord(G) \times \ord(G)}(\xi) = \lambda$. We note that $\xi$ is a $G$-invariant measure on $X \times \pord(G) \times \ord(G)$ that is concentrated on $X \times \ordext(G)$ (since $\lambda = \proj_{\pord(G) \times \ord(G)}(\xi)$ is concentrated on $\ordext(G)$). We conclude, by Lemma \ref{lem: G E invariant equiv}, that $\xi$ is the required measure on $\ordext(X, E)$.

\end{proof}

\begin{theor}[Gaboriau-Lyons, \cite{GL09}]\label{thm: gaboriau-lyons}
For any non-amenable group $G$, the orbit equivalence relation of the Bernoulli action $G \acts [0,1]^G$ has an orbit equivalence relation of an essentially free pmp action of the free group $F_2$ as a subequivalence relation.
\end{theor}

We are ready to prove that nonamenable groups have no IRO-extension property.

\begin{proof}[Proof of Theorem \ref{thm: main group}]
Let $E_2$ be the orbit equivalence relation of the essentially free pmp $G$-action from the Gaboriau-Lyons theorem (Theorem \ref{thm: gaboriau-lyons}) and $E_1$ be the orbit equivalence relation of the essentially free pmp action of the free group $F_2$. By Theorem \ref{thm: free no iro}, $F_2$ has no IRO-extension property, so $E_1$ has no IRO-extension property by Lemma \ref{lem: iro group and equiv}. this implies that $E_2$ has no IRO-extension property, by Lemma \ref{lem: iro equiv hereditary}, so $G$ has no IRO-extension property by Lemma \ref{lem: iro group and equiv}.
\end{proof}

\begin{lem}\label{lem: amenable equiv iro}
An ergodic amenable pmp countable equivalence relation on a standard probability space has the IRO-extension property. 
\end{lem}
\begin{proof}
By the Theorem \ref{thm: amenability}, a countable ergodic pmp amenable equivalence relation is isomorphic to the orbit equivalence relation of any essentially free pmp action of a countable amenable group. We are done by Lemma \ref{lem: iro group and equiv} together with Theorem \ref{thm: amenable gp has iro}.
\end{proof}

\begin{lem}\label{lem: equiv extesnion IRO}
Let $E$ be an ergodic pmp Borel equivalence relation on standard probability space $(X, \mu)$. Let $(Y, E_Y)$ be its measure-theoretical class-bijective extension. If $E_Y$ has the IRO extension property then so does $E$.
\end{lem}
The other direction is easy to prove as well but we wouldn't need it.
\begin{proof}
Let $p: Y \to X$ be the map that we have by the definition of the class-bijective extension.
We use Lemma \ref{lem: iro fiberwise eq full}. Let $\kappa$ be an invariant measure on $\pord(X,E)$ such that $\proj_X(\kappa) = \mu$. We want to obtain an invariant measure $\xi$ on $\ordext(X,E)$ such that $\proj_{\pord(X,E)}(\xi) = \kappa$. 
We have the following commutative diagram, dashed arrows represent lifts of measures. 
\[\begin{tikzcd}
\ordext(Y, E_Y) \ar[r, leftrightarrow] \ar[d] 
& (Y, E_Y) \otimes_{(X,E)} \ordext(X, E) \ar[r] \ar[d, "id \otimes \proj"] 
& \ordext(X,E) \ar[d]\\
\pord(Y, E_Y) \ar[r, leftrightarrow] \ar[u, bend left, dashed,"\text{assumption}"]& (Y, E_Y) \otimes_{(X,E)} \pord(X, E) \ar[r] & \pord(X,E) \ar[l, bend left, dashed, "\text{relatively-independent product}"]
\end{tikzcd}\]
The right square is trivial. For the left we remind that by definition of the class-bijective map, the restriction of $p$ is a bijection between $[y]_{E_Y}$ and $[p(y)]_E$. Thus we get a natural isomorphism between $\pord([y]_{E_Y})$ and $\pord([p(y)]_E)$. Similarly, there is a natural isomorphism between $\ordext([y]_{E_Y})$ and $\ordext([p(y)]_E)$. Now we note that $\ordext(Y, E_Y)$ could be considered as a set of tuples $(x, y, \omega, \omega')$ such that $x \in X$, $y \in Y$, $p(y) = x$, $\omega \in \pord([y]_{E_Y})$ and $\omega' \in \ext(\omega)$. On the other hand, $(Y, E_Y) \otimes_{(X,E)} \ordext(X, E)$ is the set of all tuples $(x, y, \omega, \omega')$, such that $x \in X$, $y \in Y$, $p(y) = x$, $\omega \in \pord([x]_E)$ and $\omega' \in \ext(\omega)$. Now the isomorphism between $\ordext(Y, E_Y)$ and $(Y, E_Y) \otimes_{(X,E)} \ordext(X, E)$ is quite trivial after identifying $\pord([y]_{E_Y})$ with $\pord([x]_E)$. The vertical projections simply erase the $\omega'$ coordinate from the tuple. 

We equip $(Y, E_Y) \otimes_{(X,E)} \pord(X, E)$ with the relatively-independent product-measure and push it to obtain an invariant measure $\lambda$ on $\pord(Y, E_Y)$. Using the IRO-extension property for $E_Y$, we lift that measure to an invariant measure $\gamma$ on $\ordext(Y, E_Y)$. We push it into an invariant measure $\xi$ on $\ordext(X,E)$. By the construction and commutativity of the diagram, $\proj_{\pord(X,E)}(\xi) = \kappa$. 
\end{proof}

\begin{lem}\label{lem: non-amenable equiv no iro}
An ergodic non-amenable countable Borel pmp equivalence relation on a standard probability space has no IRO-extension property.
\end{lem}
\begin{proof}
By the main theorem of \cite{BHI18}, any ergodic non-amenable countable pmp Borel equivalence relation 
has a class-bijective measure-theoretical extension (so-called Bernoulli extension) that contains an orbit equivalence relation of an essentially free pmp action of the free group $F_2$ as a subequivalence relation. By Theorem \ref{thm: free no iro} and Lemmata \ref{lem: iro group and equiv}, \ref{lem: iro equiv hereditary}, the said extension has no IRO-extension property, using Lemma \ref{lem: equiv extesnion IRO}, we conclude that our initial equivalence relation has no IRO-extension property.
\end{proof}

Lemmata \ref{lem: amenable equiv iro} and \ref{lem: non-amenable equiv no iro} together imply that an ergodic countable pmp Borel equivalence relation on a standard probability space has the IRO-extension property iff it is amenable, thus proving Theorem \ref{thm: equiv IRO}.

\end{document}